\documentclass[12pt, reqno]{amsart}
\usepackage{amsmath, amsthm, amscd, amsfonts, amssymb, graphicx, color}
\usepackage[bookmarksnumbered, colorlinks, plainpages]{hyperref}
\hypersetup{colorlinks=true,linkcolor=red, anchorcolor=green, citecolor=cyan, urlcolor=red, filecolor=magenta, pdftoolbar=true}
\usepackage{tikz}
\usetikzlibrary{arrows,decorations.pathmorphing,backgrounds,positioning,fit}

\newtheorem{theorem}{Theorem}[section]
\newtheorem{lemma}[theorem]{Lemma}
\newtheorem{proposition}[theorem]{Proposition}
\newtheorem{corollary}[theorem]{Corollary}
\theoremstyle{definition}

\newtheorem{example}[theorem]{Example}

\theoremstyle{remark}
\newtheorem{remark}[theorem]{Remark}
\numberwithin{equation}{section}

\begin{document}

\setcounter{page}{1}

\title[On Lax-Phillips scattering matrix of the abstract wave equation]{On Lax-Phillips scattering matrix of the abstract wave equation}

\author[M. Gawlik,  A. G\l\'{o}wczyk \MakeLowercase{and}  S. Kuzhel ]{M. Gawlik,$^1$ A. G\l\'{o}wczyk$^1$ \MakeLowercase{and} S. Kuzhel$^1$$^{*}$}

\address{$^{1}$AGH University of Science and Technology, Krak\'{o}w 30-059, Poland.}
\email{\textcolor[rgb]{0.00,0.00,0.84}{mgawlik@onet.com.pl}}
	
\email{\textcolor[rgb]{0.00,0.00,0.84}{annaglowczyk@vp.pl}}

\email{\textcolor[rgb]{0.00,0.00,0.84}{kuzhel@agh.edu.pl}}






\begin{abstract}
The dependence of singularities of scattering matrices of the abstract wave equation on the choice of
asymptotically equivalent outgoing/incoming subspaces is studied. The obtained results are applied to
the radial wave equation with nonlocal potential. In the latter case, the concept of associated inner function
introduced in  the Douglas-Shapiro-Shields work \cite{DSS} plays an essential role.
\end{abstract} 

\maketitle

\section{Introduction}\label{intro}
A  continuous group of unitary operators $W(t)$ acting in a Hilbert space ${H}$
is a subject of the Lax-Phillips scattering theory \cite{LF} 
if there exist so-called \emph{incoming} $D_-$ and \emph{outgoing} $D_+$ subspaces of ${H}$ with properties:
$$
\begin{array}{l}
(i) \quad W(t)D_+\subset{D_+}, \qquad W(-t)D_-\subset{D_-}, \quad t\geq{0}; \vspace{3mm} \\
(ii) \quad \bigcap_{t>0}W(t)D_+=\bigcap_{t>0}W(-t)D_-=\{0\}; \vspace{3mm} \\
(iii) \quad \bigvee_{t\in\mathbb{R}}W(t)D_+=\bigvee_{t\in\mathbb{R}}W(-t)D_-=H.
\end{array}
$$

Conditions $(i)-(iii)$ allow to construct incoming and outgoing 
spectral representations  $L_2({\mathbb R}, N)$ for $W(t)$  \cite[p. 50]{LF}
and define the corresponding Lax--Phillips scattering matrix ${S}(\delta)$ \
$(\delta\in{\mathbb R})$ whose values are unitary operators in $N$. 
Furthermore, the additional  condition of orthogonality 
$$
 (iv) \qquad  D_+\bot{D_-}
$$ 
guarantees that ${S}(\delta)$  is the boundary value of a contracting 
 operator-valued function ${S}(z)$ holomorphic in the lower half-plane $\mathbb{C}_-$ \cite[p. 52]{LF}.

A point $z\in\mathbb{C}_-$ is called a \emph{singularity point} of ${S}(\cdot)$ if
$0\in\sigma(S(z))$. The singularities of  ${S}(\cdot)$ are closely related to the behavior
of the semigroup $Z(t)=PW(t)P$, where $P$ is the orthogonal projection operator on 
$D_-\oplus{D_+}$ in $H$.  Since the subspaces $D_{\pm}$ characterize a free evolution in
the Lax--Phillips scattering theory, the semigroup $Z(t)$  expresses the influence of perturbation encoded in $W(t)$.

The properties above (holomorphic continuation and the relationship between the singularities of ${S}(\cdot)$ and the perturbation) 
are characteristic for the Lax-Phillips approach in scattering theory. 

The incoming and outgoing subspaces are not determined uniquely and their choice
must be consistent with the specifics of the problem.  For example, for a given incoming subspace
$D_-$, the subspace $D_+={H}\ominus{D_-}$ turns out to be outgoing. In this case 
 the corresponding Lax--Phillips scattering matrix is the identity operator and, obviously,  it has no singularity points. 
This simple example illustrates the importance of a proper choice of incoming and outgoing subspaces
for constructing nontrivial scattering matrices.

Following \cite[p. 87]{LF} we say that  subspaces $D$ and $D'$ are
 \emph{equivalent} with respect to
 $W(t)$ if there exists $a\in\mathbb{R}$ such that
$$
W(a)D\subset{D'} \quad \mbox{and} \quad W(a)D'\subset{D}.
$$

 For each equivalent orthogonal outgoing/incoming subspaces
$D_{\pm}$ and $D'_{\pm}$, the holomorphic continuations $S(z)$ and
$S'(z)$ of the associated Lax--Phillips scattering matrices
$S(\cdot)$ and $S'(\cdot)$ are related as follows:
$$
S'(z)={\frak M}_+(z)S(z){\frak M}_-^{-1}(z), \quad z\in\mathbb{C}_-, 
$$ 
where
${\frak M}_\pm(z)$  are \emph{trivial inner factors} \cite[p. 88, 89]{LF}.
The last relation means that $S(\cdot)$ and $S'(\cdot)$ have the same sets of singularity points in $\mathbb{C}_-$.
Therefore, the choice of equivalent outgoing/incoming subspaces
does not change the singularities of Lax-Phillips scattering matrices.

In the present paper,  \emph{we investigate how the set of singularities 
is changed under the choice of non-equivalent  outgoing/incoming  subspaces.} 
We focus our attention on the case where $W(t)$ is the group of solutions of the Cauchy problem  
for an abstract realization of the classical wave equation (\emph{abstract wave equation}) and subspaces
$D$ and $D'$ are \emph{asymptotically equivalent} (see \eqref{e63} ).  Precisely, we consider 
an evolving system described by an operator-differential equation
 \begin{equation}\label{e2}
                u_{tt}=-Lu,
 \end{equation}
 where $L$ is a positive self-adjoint operator in a Hilbert space  ${\mathfrak H}$.
 
Denote by  ${\mathfrak H}_{L}$  the Hilbert space which is the completion of the domain 
$\mathcal{D}(L)$ with respect to the norm $\|{u}\|_{L}^2:=(L{u},u)$.
In the energy space 
\begin{equation}\label{AGH32}
H={\mathfrak H}_{L}\oplus{\mathfrak H}=\left\{\left[\begin{array}{c}
u \\
v 
\end{array}\right] \ : \   u\in\mathfrak{H}_L, \quad v\in\mathfrak{H}\right\}
\end{equation}
 equation \eqref{e2}
determines a group of unitary operators
$W(t)$  -- solutions of the Cauchy problem \cite[p. 53]{LF2}.

When $L=-\Delta$ and $\mathfrak{H}=L_2(\mathbb{R}^n)$ ($n$ is odd), the expression
\eqref{e2} gives the  wave equation  $u_{tt}=\Delta{u}$ in a space of odd dimension.
The corresponding classical outgoing/incoming subspaces $D_\pm$ constructed in  \cite{LF} 
possess the additional property  
\begin{equation}\label{e15}
(v) \quad JD_-=D_+,
\end{equation}
where $J$ is a self-adjoint and unitary operator in $H$ (so-called time-reversal operator):
\begin{equation}\label{AGH3}
J\left[\begin{array}{c}
u \\
v 
\end{array}\right]=\left[\begin{array}{c}
u \\
-v 
\end{array}\right].
\end{equation}
Note that relation \eqref{e15}  illustrates the `equal rights' of the incoming $D_-$ and the outgoing $D_+$ subspaces 
  with respect to the  time-reversal operator $J$ and it is a characteristic property of dynamics governed by  
 wave equations \cite{LF, LF2}.

Obviously, the existence of outgoing/incoming subspaces for the group
of solutions of Cauchy problem should be related with specific properties of the operator $L$
in \eqref{e2}.  Before formulating the result explaining which properties of $L$ are needed, we
remark that not all Lax-Phillips conditions $(i)-(iv)$ are equally significant.
In particular, if $W(t)$ satisfies $(i), (ii)$, and $(iv)$, then 
the restrictions of $W(t)$ onto
\begin{equation}\label{AGH21}
M_-={\bigvee_{t\in{\mathbb R}}W(t)D_{-}} \quad \mbox{and} \quad
M_+={\bigvee_{t\in{\mathbb R}}W(t)D_{+}},
\end{equation}
have, respectively, incoming and outgoing spectral representations and the corresponding
scattering matrix ${S}(\cdot)$ admits a holomorphic  continuation in $\mathbb{C}_-$.
The set of singularities of ${S}(\cdot)$ in $\mathbb{C}_-$ is defined as  above.
The difference with the previous case, consists only in the fact that  ${S}(\delta)$ 
are contraction operators \cite{Arov, KK}. 

We recall that: a symmetric operator is called \emph{simple}
 if its restriction on any nontrivial reducing subspace is not a
 self-adjoint operator;  \emph{the maximality} of a symmetric operator means that
 one of its defect numbers is zero. 
 
\begin{theorem}[\cite{KK, KU2}]\label{AGH10}
Let $W(t)$ be  the group of solutions of the Cauchy problem of \eqref{e2}
and let subspaces  $D_\pm\subset{H}$ satisfy conditions $(i), (ii), (iv)$, and $(v)$.
Then there exists  a simple maximal symmetric operator $B$ acting in a subspace ${\mathfrak
H}_0$ of the Hilbert space ${\mathfrak H}$ such that the operator $L$ is a positive self-adjoint extension
(with exit in the space ${\mathfrak H}$) of the symmetric operator $B^2$

The subspaces $D_+$ and $D_-$ coincide with the closures (in the energy space $H$) 
of the following sets:
\begin{equation}\label{e2b}
\left\{\left[\begin{array}{c}
u \\ iBu
\end{array}\right] \ \left|\right. \ \forall{u}\in\mathcal{D}(B^2) \right\} \quad
\mbox{and} \quad
\left\{\left[\begin{array}{c}
u \\
-iBu
\end{array}\right] \ \left|\right. \ \forall{u}\in\mathcal{D}(B^2) \right\},
\end{equation}
respectively (without loss of generality, we assume that
$B$ has  zero defect number in $\mathbb{C}_+$). Moreover, for all $t\geq{0}$,
\begin{equation}\label{e3b}
W(t)\left[\begin{array}{c}
u \\ iBu
\end{array}\right]=\left[\begin{array}{c}
V(t)u \\ iBV(t)u
\end{array}\right], \qquad  W(-t)\left[\begin{array}{c}
u \\ -iBu
\end{array}\right]=\left[\begin{array}{c}
V(t)u \\ -iBV(t)u
\end{array}\right],  
\end{equation}
where $V(t)=e^{iBt}$ is a semigroup of isometric operators in $\mathfrak{H}_0$.
 
 Conversely, if there exists a simple maximal symmetric operator $B$ acting in  $\mathfrak{H}_0\subseteq\mathfrak{H}$
 and such that $L$ is an extension of $B^2$, then the subspaces $D_\pm$ defined by \eqref{e2b} are outgoing/incoming for 
 $W(t)$ (i.e., conditions $(i), (ii), (iv)$, $(v)$ hold) and \eqref{e3b} remains true.
 \end{theorem}

Considering variuos operators $B$ in (\ref{e2}) leads to different pairs of outgoing/incoming subspaces.
For instance, if ${\mathfrak H}_0=L_2({\mathbb R}^n) \ \ (n\geq{3} \ \mbox{is odd})$  
and 
\begin{equation}\label{AGH25}
B=\Xi^{-1}_+i\frac{d}{ds}\Xi_+, \qquad  \mathcal{D}(B)=\Xi^{-1}_+\{u\in{{W}_{2}^{1}}({\mathbb R}_+, N) : u(0)=0\},
\end{equation}
where  $N=L_2(S^{n-1})$ is the Hilbert space of functions square-integrable on
the unit sphere $S^{n-1}$ in ${\mathbb R}^n$ and the isometric operator
$\Xi_+ : L_2({\mathbb R}^n)\to{L_2({\mathbb R}_+,N)}$ is defined on
rapidly decreasing smooth functions $u(x)\in{S({\mathbb R}^n)}$ as:
$$
(\Xi_+{u})(s,w)=(\partial_s^mRu)(s,w) \qquad (m=\frac{(n-1)}{2}, \ s\geq{0}, \
w\in{S^{n-1}}), 
$$ 
where $R$ is the Radon transform, then the formulas \eqref{e2b} give  
the classical Lax--Phillips  subspaces $D_{\pm}$ for the free
wave equation in ${\mathbb R}^n$, which were described in \cite[Chapter IV]{LF}.

Following \cite{MFAT}, we say that subspaces $D$ and $D'$ are 
\emph{asymptotically equivalent} (quasi-equivalent) with respect to $W(t)$ if   
\begin{equation}\label{e63}
D\subset{\bigvee_{t\in{\mathbb R}}W(t)D'} \quad \mbox{and} \quad
 D'\subset{\bigvee_{t\in{\Bbb R}}W(t)D}.
\end{equation}

Obviously, each equivalent subspaces are asymptotically equivalent. The inverse
statement is not true. 
Asymptotically equivalent subspaces are studied in Section \ref{sec2}. 
The main attention is paid to the case where $D_{\pm}'$ are subspaces of $D_{\pm}$. This condition
fits well the specific of wave equation (see  \cite{AlAn} and \cite[p. 142]{LF}  for the relevant discussion) and
it can be realized as follows: the formula  \eqref{e2b} describes  simultaneously $D_{\pm}$ and $D_{\pm}'$  but the `bigger' subspaces $D_{\pm}$ 
are determined by a simple maximal symmetric operator $B$ acting in $\mathfrak{H}_0$, while the `smaller' subspaces $D_{\pm}':=D_\pm^V$ are
described by the new simple maximal symmetric operator $B_V$ (defined by \eqref{e7}) which is the restriction of $B$ onto a subspace
$\mathfrak{H}_0^V\subset\mathfrak{H}_0$ see \eqref{new1}.  

The scheme above is well defined when the isometric operator $V$ in the definition of $\mathfrak{H}_0^V$ commutes with $B$. 
For this reason, it is natural to consider $V$ as a function of $B$.  The functional calculus 
for maximal symmetric operators  was proposed by Plesner in series of short papers in russian \cite{Ples1}  -\cite{Ples3} 
without proofs. To the best of our knowledge, these papers have not been translated.
For the reader's convenience, we prove some results in the Appendix. In particular, we show that each inner function  
$\psi\in{H^\infty(\mathbb{C}_+)}$ determines an isometric operator $V=\psi(B)$ which 
commutes with $B$. The main result of Section 2 states that the subspaces
 $D_\pm$ and $D_\pm^{\psi(B)}$ are asymptotically equivalent  
(Proposition \ref{AGH55b}).

Let $V=\psi(B)$ and let $S(\cdot)$ and $S_V(\cdot)$  be the Lax--Phillips scattering matrices for the pairs
$D_{\pm}$ and $D_{\pm}^V$,  respectively. 
In Section \ref{sec3}, we show that $S_V(\cdot)$ has new points of singularity $-\lambda$ and $\overline{\lambda}$ in $\mathbb{C}_-$
which are determined by zeros $\lambda\in\mathbb{C}_+$ of $\psi$. 
Therefore, in contrast to the case of equivalent  subspaces, the choice 
of asymptotically equivalent subspaces $D_\pm$ and $D_\pm^V$ may lead 
to the  appearance of `false' zeros of $S_V(\cdot)$ which are not related to the specific of perturbation
and caused only by the choice of $V=\psi(B)$.

In Section \ref{sec4}, the radial wave equation with nonlocal potential $f(\cdot, f)$, where
$f\in{L_2(\mathbb{R}_+)}$ is considered. We show that the `bigger' subspaces $D_\pm$ are 
 constructed by the inner function  $\psi_0(\delta)=\phi\left(\frac{\delta-i}{\delta+i}\right)$,
where $\phi$ is the associated inner function of the isometric transformation $\gamma$ of the function $f$ in $H^2(\mathbb{D})$.
As was mention in the well-known Douglas-Shapiro-Shields work \cite[Remark 3.1.6]{DSS}, 
the function $\phi$ is uniquely determined by $\gamma$ in the decomposition  \eqref{AGH48} 
and it plays a role in the study of the left shifts of $\gamma$ completely analogous 
to the role which the inner factor of $\gamma$ plays in the study of the right shifts. For this reason we can expect that the singularities of
$S(\cdot)$ associated with $D_{\pm}$ correspond to the influence of nonlocal potential  $f(\cdot, f)$ in the right way.

The `smaller' subspaces $D_{\pm}^V$ are constructed by an inner function $\psi_1$ which is divisible by $\psi_0$.
The subspaces $D_{\pm}^V$ are asymptotically equivalent with $D_{\pm}$  and the corresponding scattering matrix $S_V(\cdot)$ 
may have additional singularities generated by zeros of the function  $\psi={\psi_1}/{\psi_0}$ in $\mathbb{C}_+$
which have no relation to the nonlocal potential  $f(\cdot, f)$.

Throughout the paper, $\bigvee_{t\in\mathbb{R}}X_t$ means the closure of linear span of sets $X_t$,  $\mathcal{D}(A)$ denotes the
domain of a linear operator $A$. The symbols $H^p(\mathbb{D})$ and $H^p(\mathbb{C}_+)$ are used for  the Hardy spaces 
in $\mathbb{D}=\{\lambda\in\mathbb{C} : |\lambda|<1\}$ and $\mathbb{C}_+=\{z\in\mathbb{C} : Im
 \ z >0 \}$, respectively.  The  Sobolev space is denoted as $W_2^p(I)$ ($I\in\{\mathbb{R}, \mathbb{R}_\pm \}$, $p\in\{1, 2\}$).
 The symbol $N$ is used for an auxiliary Hilbert space. The notations $H^p(\mathbb{D}, N)$ and $H^p(\mathbb{C}_+, N)$, and
 $W_2^p(I, N)$ are used for the Hardy and Sobolev spaces of vector functions with values in $N$. 

\section{Asymptotically equivalent subspaces}\label{sec2}

\subsection{Preliminaries. Outgoing and incoming  spectral representations}\label{s1}
Let $B$ be a densely defined symmetric operator in a Hilbert space $\mathfrak{H}_0$ with inner product $(\cdot,\cdot)$ linear in the first
argument. 
The defect numbers of $B$ in $\mathbb{C_\pm}$ are defined as $\dim\ker(B^*\mp{i}I)$, where
$B^*$ is the adjoint of $B$.

 A symmetric operator $B$ is called \emph{simple}  if it does not induce a self-adjoint
operator in any proper subspace of ${\mathfrak H}_0$ and $B$ is called \emph{maximal symmetric} if one of its defect numbers 
is equal to zero. 

Let us suppose that a simple maximal symmetric operator $B$ has  zero defect number in $\mathbb{C}_+$. 
Then there exists an isometric operator $\Xi_+ : {\mathfrak H}_0 \to L_{2}({\mathbb R}_{+}, N)$, 
 such that
    \begin{equation}\label{e14}
 B=\Xi_+^{-1}i\frac{d}{dx}\Xi_+, \qquad  {\mathcal{D}}(B)=\Xi_+^{-1}\{u\in{{W}_{2}^{1}}({\mathbb R}_+, N) : u(0)=0\},   
\end{equation}
where the dimension of the auxiliary Hilbert space $N$ is equal to  $\dim\ker(B^*+{i}I)$ \cite[ $\S$ 104]{AG}.
An example of such kind of isometric mapping is given in \eqref{AGH25}.
 Similarly, if $B$ is a simple maximal symmetric operator with zero defect number in $\mathbb{C}_-$, 
 then there exists an isometric operator $\Xi_- : {\mathfrak H}_0 \to L_{2}({\mathbb R}_{-}, N)$ 
 such that
\begin{equation}\label{e14b}
 B=\Xi_-^{-1}i\frac{d}{dx}\Xi_-, \qquad  {\mathcal{D}}(B)=\Xi_-^{-1}\{u\in{{W}_{2}^{1}}({\mathbb R}_-, N) : u(0)=0\},   
\end{equation}
 where the dimension of $N$ is equal to  $\dim\ker(B^*-{i}I)$. 
 
Let $W(t)$ be a group of solutions of Cauchy problem of the abstract wave equation \eqref{e2}
and let the subspaces  $D_\pm\subset{H}$ satisfy conditions $(i), (ii), (iv)$, and $(v)$.
In what follows, without loss of generality we assume that $B$ has zero defect number in $\mathbb{C}_+$.
Then the formula \eqref{e14}  and the Fourier transformation in $L_2({\mathbb R},N)$: 
  $$
 Ff(\delta)=\frac{1}{\sqrt{2\pi}}\int^\infty_{-\infty}e^{i\delta{s}}f(s)ds, \qquad f\in{L}_2({\mathbb R}, N)    
 $$
  allow us to obtain an explicit formula for the spectral representations for $W(t)$ associated with $D_\pm$.
 We briefly recall principal formulas that are necessary for our presentation (see \cite[Chapter 4]{KK} for detail).

According to Theorem \ref{AGH10}, the subspaces $D_\pm$ are determined by \eqref{e2b} with a simple maximal symmetric operator $B$ in 
$\mathfrak{H}_0$. Denote $L_\mu=B^*B$. The operator $L_\mu$ acting in $\mathfrak{H}_0$ is a positive self-adjoint extension of $B^2$
(moreover $L_\mu$ is the Friedrichs extension  of $B^2$).

Let $W_\mu(t)$   be a group of solutions of Cauchy problem of \eqref{e2} with the operator $L_\mu$ in the right-hand side. In this case,
the corresponding energy space $H_\mu={\mathfrak H}_{L_\mu}\oplus{\mathfrak H}_0$ coincides with $D_-\oplus{D_+}$ and
it can be considered as a subspace of the energy space $H$ defined in \eqref{AGH32}.

Since relations \eqref{e3b}  hold simultaneously  for $W_\mu(t)$ and for $W(t)$,   the wave operators
$\Omega_\pm=s-\lim_{t\to\pm\infty}W(-t)W_\mu(t)$ exist and they isometrically map $H_\mu=D_-\oplus{D_+}$ onto $M_\pm$, 
where $M_\pm$ are defined in  \eqref{AGH21}. 
Furthermore,
\begin{equation}\label{AGH51} 
\Omega_\pm{W_\mu(s)}=W(s)\Omega_\pm, \quad s\in\mathbb{R},  \qquad  \Omega_\pm{d_\pm}=d_{\pm},  \quad \forall{d}_\pm\in{D_\pm}.
\end{equation}

Consider the mapping
\begin{equation}\label{AGH31}
G\left[\begin{array}{c}
u \\
v \end{array}
\right]=\frac{1}{\sqrt{2}}\left\{\begin{array}{l}
\Xi_+(iBu+v)(s) \quad (s>0) \\
\Xi_+(iBu-v)(-s) \quad (s<0)
\end{array}\right. \quad u\in\mathcal{D}(B^2), \quad v\in\mathfrak{H}_0,
\end{equation}
where $\Xi_+$ is taken from \eqref{e14}. Since,
$$
\left\|G\left[\begin{array}{c}
u \\
v \end{array}
\right]\right\|^2_{L_2(\mathbb{R}, N)}=\|Bu\|^2+\|v\|^2=\left\|\left[\begin{array}{c}
u \\
v \end{array}
\right]\right\|^2_H 
$$
the operator $G$ can  be extended by the continuity (in $H_\mu$) to
an isometric mapping of $H_\mu=D_-\oplus{D_+}$ onto $L_2(\mathbb{R}, N)$ and
such that $GD_{\pm}=L_2(\mathbb{R}_\pm, N)$.
Moreover, 
\begin{equation}\label{AGH53}
GW_\mu(t)d=\mathcal{T}(t)Gd, \qquad \forall{d}\in{D_-\oplus{D_+}}, 
\end{equation}
where $\mathcal{T}(t)f(x)=f(x-t)$ is the translation to the right by $t$ units in $L_2(\mathbb{R}, N)$
\cite[p. 221]{KK}.

It follows from  \eqref{AGH51}, \eqref{AGH31}, and \eqref{AGH53} that the operators
\begin{equation}\label{AGH40}
R_+=FG\Omega_+^{-1}  : M_+ \to L_2(\mathbb{R}, N), \qquad  R_-=FG\Omega_-^{-1}  : M_- \to L_2(\mathbb{R}, N)
\end{equation}
map $D_+$ and $D_-$ onto $H^2(\mathbb{C}_+, N)$ and $H^2(\mathbb{C}_-, N)$,  respectively and they
define outgoing/incoming spectral representations $L_2(\mathbb{R}, N)$ for the
restrictions of $W(t)$ onto $M_\pm$.

\subsection{Asymptotically equivalent subspaces.}\label{s2}
Let $W(t)$ be a group of solutions of Cauchy problem of  \eqref{e2}
and let  $D_\pm$ and $D_\pm'$ be different pairs of subspaces that satisfy conditions $(i), (ii), (iv)$, and $(v)$.
\begin{lemma}
If  $D_+$ and $D_+'$ are asymptotically equivalent  then $D_-$ and $D_-'$ 
are asymptotically equivalent and vice-versa.
\end{lemma}
\begin{proof}
 Denote by $iQ$ the generator of $W(t)$. The operator $Q$ is self-adjoint in $H$ and it coincides  with the closure of the operator \cite[p.55]{LF2}
$$
\mathcal{Q}= i\left[\begin{array}{cc} 0 &  -I  \\  
L  &  0\end{array} \right],  \qquad   \mathcal{D}(\mathcal{Q})=\left\{\left[\begin{array}{c} u \\
v
\end{array}\right] \  |  \  u, v\in\mathcal{D}(L) \right\}. 
$$
In view of \eqref{AGH3}, $J\mathcal{Q}=-\mathcal{Q}J$. Therefore, $J$ anticommutes with $Q$ and
\begin{equation}\label{AGH71}
JW(t)=W(-t)J.
\end{equation}
By virtue of \eqref{e15} and \eqref{AGH71}, the inclusion $D_+ \subset\vee_{t \in \mathbb{R}}W(t) D_+'$ implies that 
$$
D_-=JD_+\subset\bigvee_{t \in \mathbb{R}}W(-t)JD_+'=\bigvee_{t \in \mathbb{R}}W(t)D_-'.
$$
The second inclusion in \eqref{e63} is transformed similarly.  
\end{proof}

By Theorem \ref{AGH10}, there exist simple maximal symmetric operators $B$ and $B'$ acting in subspaces 
 ${\mathfrak H}_0$ and ${\mathfrak H}_0'$ of ${\mathfrak H}$ and such that 
 the subspaces $D_\pm$ and $D_\pm'$ are determined by  \eqref{e2b} with $B$ and $B'$,
 respectively.

\begin{lemma}\label{lem2} 
If $D_+$ and $D_+'$ are asymptotically equivalent, then the corresponding operators $B$ and $B'$
are unitary equivalent.
\end{lemma}
\begin{proof}
 Since  $D_+$ and $D_+'$ are asymptotically equivalent, relations \eqref{e63}
imply that the subspace $M_{+}$ in \eqref{AGH21} coincides with 
${\bigvee_{t\in{\mathbb R}}W(t)D_{+}'}$.  Hence, outgoing spectral representations $L_2(\mathbb{R}, N)$ and 
$L_2(\mathbb{R}, N')$ for the group $W(t)$ associated with $D_{+}$ and $D_{+}'$, respectively are constructed for the restriction
of $W(t)$ onto the same subspace $M_+$. This gives $\dim{N}=\dim{N'}$ since the auxiliary spaces in spectral representations are determined uniquely
up to isometries \cite[p. 50]{LF}.  The auxiliary spaces can be taken from the expression \eqref{e14} for $B$ and $B'$.
This means that  $B$ and $B'$ are unitary equivalent.
\end{proof}

In what follows we consider the case where $D_{\pm}'$ are subspaces of $D_{\pm}$. Such kind of relation is typical
for the wave equation (see \cite{AlAn}, \cite[p. 142]{LF}).

Since the subspaces $D_{\pm}'$ and ${D_\pm}$ are described by \eqref{e2b} with operators $B'$ and $B$, respectively  
the inclusions $D_{\pm}'\subset{D_\pm}$ can be easy realized assuming that $\mathfrak{H}_0'\subset{\mathfrak{H}_0}$ and
 $B'$ is the part of $B$ restricted on  $\mathfrak{H}_0'$. Moreover,  due to Lemma \ref{lem2},  $B'$ should be unitary equivalent to $B$
 (if we are going to investigate asymptotically equivalent subspaces). Therefore, we have to suppose the 
 existence of an isometric operator $V$ acting in a Hilbert space ${\mathfrak
H}_0$ and such that:
$$
 \mathfrak{H}_0'=V\mathfrak{H}_0, \quad \mbox{and} \quad  VBu=B'Vu=BVu, \quad \forall{u}\in\mathcal{D}(B).
$$
 
Summing up,  in what follows,  we will assume  that $\mathfrak{H}_0'=V\mathfrak{H}_0:=\mathfrak{H}_0^V$,
 where $V$ is an isometric operator in ${\mathfrak H}_0$  that commutes with $B$
\begin{equation}\label{a68}
VBu=BVu, \qquad \forall{u}\in\mathcal{D}(B).
\end{equation}
 Then the operator
\begin{equation}\label{e7}
B_V:=B'=VBV^*, \qquad  \mathcal{D}(B_V)=\mathcal{D}(B')=V\mathcal{D}(B),
\end{equation}
is simple maximal symmetric in the Hilbert space ${\mathfrak H}_0^V$. 

  It follows from (\ref{a68}) and (\ref{e7}) that
\begin{equation}\label{new1}
 \mathcal{D}(B_{V})=\mathcal{D}(B)\cap\mathfrak{H}_0^{V}   \quad \mbox{and} \quad  B_{V}u=Bu, \quad \forall{u}\in\mathcal{D}(B_{V}),
 \end{equation}
 (i.e., $B_V$ is a part of $B$ restricted on $\mathfrak{H}_0^{V}$). Moreover, 
 $$
 B_{V}^*u={V}B^*{V}^*u,  \qquad \forall{u}\in\mathcal{D}(B_{V}^*)=V\mathcal{D}(B^*).   
 $$

 By virtue of \eqref{new1},  the operator $L$ in \eqref{e2} is an extension of $B_V^2$
 (since $L$ is an extension of $B^2$ by the assumption). 
 Therefore,  by Theorem \ref{AGH10},  the subspaces $D_\pm^V:=D_\pm'$ determined by \eqref{e2b} (with  $B_V$ instead of $B$)
 are outgoing/incoming for $W(t)$ and conditions $(i), (ii), (iv)$, $(v)$ hold.
 
 In general, we can not state that $D_\pm$ and $D_\pm^V$ are asymptotically equivalent. 

\begin{theorem}\label{AGH14}\cite{MFAT}
The subspaces $D_\pm$ and $D_\pm^V$ are asymptotically equivalent if and only if 
\begin{equation}\label{AGH73}
\lim_{t\to+\infty}\|P_{{\ker {V}^*}}e^{iBt}\gamma\|=0 \qquad \forall\gamma\in\ker{V}^*=\mathfrak{H}_0\ominus\mathfrak{H}_0^V,
\end{equation}
where $P_{\ker {V}^*}$ is the orthogonal projection in  ${\mathfrak
H}_0$ on $\ker {V}^*$.
 \end{theorem}
 
 \begin{corollary}
 If $\dim(\mathfrak{H}_0\ominus\mathfrak{H}_0^V)<\infty$, then $D_\pm$ and $D_\pm^V$ are asymptotically equivalent.
 \end{corollary}
 \begin{proof}
 It follows from the fact that $P_{{\ker {V}^*}}$ is a compact operator in $\mathfrak{H}_0$ and $e^{iBt}\gamma\to{0}$ in the sense of weak convergence.
 \end{proof}
 
 The Cayley transform of a simple maximal symmetric operator $B$ 
\begin{equation}\label{AGH2}
T=(B-iI)(B+iI)^{-1}
\end{equation}
is a unilateral shift in $\mathfrak{H}_0$.  It is useful to rewrite Theorem \ref{AGH14} in terms of $T$.

\begin{corollary}\label{AGH55}
 The subspaces $D_\pm$ and $D_\pm^V$ are asymptotically equivalent if and only if 
\begin{equation}\label{AGH73b}
\lim_{n\to+\infty}\|(P_{{\ker {V}^*}}T)^n\gamma\|=0 \qquad \forall\gamma\in\ker{V}^*.
\end{equation}
\end{corollary}
\begin{proof}
The equivalence  between \eqref{AGH73} and \eqref{AGH73b} is a `folklore result' of operator theory. 
We outline principal stages of the proof.  The operator-valued function
$K(t)=P_{{\ker {V}^*}}e^{iBt}P_{{\ker {V}^*}}$ $(t\geq{0})$ is a semigroup of contraction operators in 
 $\mathfrak{H}_0\ominus\mathfrak{H}_0^V$. Let $K$ be its cogenerator. Then \cite[p. 150, formula (9.18)]{SNK},
 $$
  \lim_{t\to+\infty}\|P_{{\ker {V}^*}}e^{iBt}\gamma\|=\lim_{t\to+\infty}\|K(t)\gamma\|=\lim_{n\to+\infty}\|K^n\gamma\|, \qquad \gamma\in\ker {V}^*.
 $$
 Taking into account that  $K=P_{{\ker {V}^*}}T$, where $T$ is defined by
  \eqref{AGH2}  (it follows from \cite[p. 144, formula (8.8)]{SNK}) we complete the proof.
 \end{proof}

\begin{proposition}\label{AGH55b}
If $V=\psi(B)$, where $\psi\in{H^\infty(\mathbb{C}_+)}$ is an inner function,
then the subspaces $D_\pm$ and $D_\pm^V$ are asymptotically equivalent.
\end{proposition}
\begin{proof}
By virtue of Corollary \ref{new4},  $V=\psi(B)$  is an isometric operator in $\mathfrak{H}_0$. 
Hence, the subspaces $D_\pm$ and $D_\pm^V$ are asymptotically equivalent if \eqref{AGH73b} holds.
 In view of \eqref{e52}, the operator $\psi(B)$ coincides with $\phi(T)$.
This means that  the subspace $\ker{V}^*$ in  \eqref{AGH73b} can be rewritten as $\ker{V}^*=\mathfrak{H}_0\ominus\phi(T)\mathfrak{H}_0$.
The unilateral shift $T$ in $\mathfrak{H}_0$ is unitary equivalent to the operator of multiplication by $\lambda$ in $H^2(\mathbb{D}, N)$, where
$N=\mathfrak{H}_0\ominus{T\mathfrak{H}_0}$ \cite[p. 198]{SNK}.  In this case, 
the subspace $\phi(T)\mathfrak{H}_0$ is transformed to $\phi{H}^2(\mathbb{D}, N)$ and the vector
$P_{{\ker {V}^*}}T\gamma$, $\gamma\in\ker{V^*}$  into the function $P\lambda{u}(\lambda)$, 
where $u\in{H}^2(\mathbb{D}, N)\ominus\phi{H}^2(\mathbb{D}, N)$ and
$P$ is an orthogonal projection in ${H}^2(\mathbb{D}, N)$ onto ${H}^2(\mathbb{D}, N)\ominus\phi{H}^2(\mathbb{D}, N)$. 
After such kind of preparatory work, \cite[Proposition 4.3 in Chapter III]{SNK} allows us to establish \eqref{AGH73b}.
\end{proof}

A proper subspace $\mathfrak{H}'$ of $\mathfrak{H}_0$ is called \emph{hyperinvariant} for $T$ if it is invariant
for each bounded operator which commutes with $T$ \cite[p. 80]{SNK}.
 
\begin{corollary}\label{AGH15}
Let $V$ be an isometric operator in $\mathfrak{H}_0$ that commutes with $B$. Then, 
the subspaces $D_\pm$ and $D_\pm^V$ are asymptotically equivalent if one of the 
following conditions hold:
\begin{itemize}
\item[(a)] the nonzero defect number of $B$ is $1$;
\item[(b)] the subspace $\mathfrak{H}_0^V$ is hyperinvariant for $T$.
\end{itemize}
\end{corollary}
\begin{proof} First of all  we note that the operator ${\mathsf Y}$ defined by \eqref{AGH1}
 maps isometrically $\mathfrak{H}_0$ onto $H^2(\mathbb{C}_+, N)$.  
 Moreover, under this mapping, the operators $B$ and $T$ are transformed 
 to the operators of multiplication
 by $\delta$ and by $\frac{\delta-i}{\delta+i}$ in $H^2(\mathbb{C}_+, N)$.

In view of  \eqref{a68}, the operator $T$ commutes with $V$. Therefore, the subspace $\mathfrak{H}_0^V$ is 
invariant for $T$.  Denote $\mathfrak{M}={\mathsf Y}\mathfrak{H}_0^V$.  Obviously, 
 $\mathfrak{M}$ is a subspace of $H^2(\mathbb{C}_+, N)$ and $\frac{\delta-i}{\delta+1}\mathfrak{M}\subset\mathfrak{M}$.  
 
 Let us prove (a). In this case
 $H^2(\mathbb{C}_+, N)=H^2(\mathbb{C}_+)$  and, by the  Beurling's theorem \cite[p. 49]{MAR},
 there exists an inner function $\psi\in{H}^\infty({\mathbb{C}_+})$  such that
$\mathfrak M=\psi{H}^2({\mathbb{C}_+})$.  Taking \eqref{new56} and Lemma \ref{AGH214} into account we arrive
at the conclusion that
$$
\mathfrak{H}_0^V={\mathsf Y}^{-1}\mathfrak M={\mathsf Y}^{-1}\psi(\delta){\mathsf Y}\mathfrak{H}_0=\psi(A)\mathfrak{H}_0=\psi(B)\mathfrak{H}_0.
$$
 Hence, without loss of generality, the isometric operator $V$ can be chosen as  $\psi(B)$. 
By Proposition \ref{AGH55b},  the subspaces $D_\pm$ and $D_\pm^V$ are asymptotically equivalent.

The proof of (b) is  similar. The difference consists in the fact that $\mathfrak{M}$ is a subspace $H^2(\mathbb{C}_+, N)$ which is  hyperinvariant with 
respect to the operator of multiplication by $\frac{\delta-i}{\delta+1}$.  A modification of the Beurling theorem
for $H^2(\mathbb{C}_+, N)$ \cite[p. 205]{SNK} implies the existence of an inner function $\psi\in{H}^\infty({\mathbb{C}_+})$ such that
$\mathfrak M=\psi{H}^2({\mathbb{C}_+}, N)$. Repeating the argumentation above we complete the proof.
\end{proof}

\section{Relation between Lax-Phillips scattering matrices associated with subspaces $D_\pm$ and $D_\pm^V$}\label{sec3}
 Let $D_{\pm}$ and $D_{\pm}^V$ be outgoing/incoming subspaces for the group $W(t)$ of solutions of Cauchy problem of  \eqref{e2}
 described in Section \ref{sec2}.  Denote by $S(\cdot)$ and $S_V(\cdot)$ the Lax--Phillips scattering matrices for the pairs
$D_{\pm}$ and $D_{\pm}^V$, respectively.

The next result was proved in \cite{MFAT} with superfluous assumption that $D_{\pm}$ and $D_{\pm}^V$ are
asymptotically equivalent.  For the convenience of the reader principal steps of the proof are repeated. 

\begin{proposition}\label{AGH17}
If  $V=\psi(B)$, where $\psi\in{H^{\infty}(\mathbb{C}_+)}$ is an inner function,  then
\begin{equation}\label{e44}
S_V(\delta)=\frac{\psi(-\delta)}{\psi(\delta)}S(\delta).
\end{equation}
\end{proposition}
\begin{proof}
By virtue of Proposition \ref{AGH55b}, the subspaces $D_\pm$ and $D_\pm^V$ are 
asymptotically equivalent. Therefore, the outgoing/incoming spectral representations associated 
with  $D_\pm$ and $D_\pm^V$, respectively, are constructed for the restrictions of $W(t)$ onto the subspaces
$M_\pm$ defined by \eqref{AGH21}.
 
The spectral representations associated with $D_\pm$ are determined by operators $R_\pm : M_\pm \to L_2(\mathbb{R}, N)$ 
in \eqref{AGH40}.   Using \eqref{AGH51} and \eqref{AGH31} we obtain
$$
e^{i\delta{t}}R_+d_+=i\sqrt{2}e^{i\delta{t}}F\left\{\begin{array}{ll} 
(\Xi_+Bu)(s) & (s>0) \\
0  & (s<0) 
\end{array}\right.  \quad \mbox{for all} \quad  d_+=\left[\begin{array}{c}
u \\ iBu
\end{array}\right]\in{D_+}.
$$
On the other hand, taking \eqref{e3b} into account,
$$
e^{i\delta{t}}R_+d_+=R_+W(t)d_+=
i\sqrt{2}F\left\{\begin{array}{ll} 
(\Xi_+V(t)Bu)(s) & (s>0) \\
0  & (s<0), 
\end{array}\right.
$$
Therefore, the operator 
$$
\Theta_+u=i\sqrt{2}F\left\{\begin{array}{ll} 
(\Xi_+Bu)(s) & (s>0) \\
0  & (s<0), 
\end{array}\right.   
$$
defined originally on $D(B^2)$ and extended by the continuity on $\mathfrak{H}_0$ maps isometrically $\mathfrak{H}_0$ onto 
$H^2(\mathbb{C}_+, N)$ and $\Theta_+V(t)=e^{i\delta{t}}\Theta_+$. This implies that the isometric mapping 
$\Theta_+$   transforms $B$  to the operator of multiplication by $\delta$ in 
$H^2(\mathbb{C}_+, N)$. Therefore,  $\Theta_+\psi(B)=\psi(\delta)\Theta_+$ and 
for elements $d_+^V=\left[\begin{array}{c}
\psi(B)u \\ iB\psi(B)u
\end{array}\right]\in{D_+^V}$ \ ($V=\psi(B)$),
$$
R_+d_+^V=\Theta_+\psi(B)u=\psi(\delta)\Theta_+u=\psi(\delta)R_+d_+.
$$

The obtained relation yields that the operator
${R}_+^V=\frac{1}{\psi(\delta)}{R}_+$
determines outgoing spectral representation for the restriction of $W(t)$ on ${M_+}$ 
which is associated with $D_+^V$.

Similarly,  the incoming spectral  representation of the restriction of $W(t)$ on ${M_-}$ associated with $D_-^V$
is determined by the operator ${R}_-^V=\frac{1}{\psi(-\delta)}R_-.$  In order to prove this formula, we
 denote 
$$
\Theta_-u=i\sqrt{2}F\left\{\begin{array}{ll} 
0 & (s>0) \\
(\Xi_+Bu)(-s)  & (s<0), 
\end{array}\right.   
$$
By analogy with the case above  $\Theta_-$ maps isometrically $\mathfrak{H}_0$ onto 
$H^2(\mathbb{C}_-, N)$ and $\Theta_-V(t)=e^{-i\delta{t}}\Theta_-$. This implies that the operator
 $B$ is transformed by $\Theta_-$ to the operator of multiplication by $-\delta$ in 
$H^2(\mathbb{C}_-, N)$. Therefore,  $\Theta_-\psi(B)=\psi(-\delta)\Theta_-$ and 
for elements $d_-^V=\left[\begin{array}{c}
\psi(B)u \\ -iB\psi(B)u
\end{array}\right]\in{D_-^V}$, we obtain
$$
R_-d_-^V=\Theta_-\psi(B)u=\psi(-\delta)\Theta_-u=\psi(-\delta)R_-d_-
$$
that justifies ${R}_-^V=\frac{1}{\psi(-\delta)}R_-.$

 The Lax--Phillips scattering matrices $S(\cdot)$ and
$S_V(\cdot)$  for the subspaces $D_{\pm}$ and $D_{\pm}^V$ are defined as
$$
S(\delta)={R}_+P_{M_+}{R}_-^{-1}, \qquad  S_V(\delta)=R_+^VP_{M_+}({R}_-^V)^{-1},
$$
 where $P_{M_+}$ is the orthogonal projection on
$M_+$ in $H$. These formulas, relations between $R_\pm$ and $R_\pm^V$ established above and
the fact that a Lax-Phillips scattering matrix commutes with
multiplication by bounded measurable functions
justify \eqref{e44}
\end{proof}

The formula \eqref{e44} holds  for every self-adjoint operator $L$ in \eqref{e2}
that satisfies conditions of Theorem \ref{AGH10}.  In particular, if we set $L=L_\mu=B^*B$,  then 
the Lax--Phillips scattering matrix $S(\cdot)$ that corresponds to subspaces $D_\pm$ coincides with the
identity operator (this fact follows from the results of Subsection \ref{s1} or \cite{AlAn}). 
In this case, \eqref{e44} gives $S_V(\delta)={\psi(-\delta)}/{\psi(\delta)}I$. Therefore, the function $\Psi(\delta)={\psi(-\delta)}/{\psi(\delta)}$ defined
on $\mathbb{R}$ is the boundary value of a holomorphic function $\Psi(z)$ in $\mathbb{C}_-$ and \eqref{e44} can be extended as follows:
\begin{equation}\label{e44c}
S_V(z)=\Psi(z)S(z), \qquad z\in\mathbb{C}_-.
\end{equation}

We recall that a point $z\in\mathbb{C}_-$ is called a \emph{singularity point} of ${S}(\cdot)$ if
$0\in\sigma(S(z))$.  Denote by $\mathfrak{S}_{S}$ the set of singularities of ${S}(\cdot)$.
 
 In view of \eqref{e44c},  $\mathfrak{S}_{S_V}=\mathfrak{S}_{S}\cup\ker\Psi$.  
Therefore, in contrast to the case of equivalent outgoing/incoming subspaces, the choice 
of asymptotically equivalent subspaces may lead 
to the  appearance of `false' zeros of $S_V$ which are not related to the abstract wave
equation  \eqref{e2} and caused only by the choice of $\psi$.

The function $\Psi(z)$ in \eqref{e44c} can be expressed in an explicit form: 
\begin{equation}\label{AGH72}
\Psi(z)=e^{-2i\alpha{z}}\prod_{n}\frac{z+\lambda_n}{z+\overline{\lambda}_n}\cdot\frac{z-\overline{\lambda}_n}{z-\lambda_n}\exp\left({-2iz\int_{\mathbb{R}}}\frac{1+t^2}{t^2-z^2}d\nu(t)\right),    \quad \alpha\geq{0},
\end{equation}
where $\lambda_n$ are the zeros of $\psi$ in $\mathbb{C}_+$ (counting multiplicities) and  $\nu$ is a finite positive singular measure on $\mathbb{R}$.

The formula \eqref{AGH72} follows from the canonical factorization of inner functions $\psi\in{H^{\infty}(\mathbb{C}_+)}$ \cite[p. 147]{Nik}. 
For example, if $\psi(\lambda)=e^{i\alpha\lambda}$ ($\lambda\in\mathbb{C}_+$),  then  ${\psi(-\delta)}/{\psi(\delta)}=e^{-2i\alpha{\delta}}$ and
 $\Psi(z)=e^{-2i\alpha{z}}$. The other cases (Blaschke product and singular inner function) are considered similarly.
 
By virtue of \eqref{e44c} and \eqref{AGH72}  the (eventually) new points of singularity of  $S_V(\cdot)$ in $\mathbb{C}_-$ coincide
with $-\lambda_n$ and $\overline{\lambda}_n$,  where $\lambda_n\in\mathbb{C}_+$ are the zeros of $\psi$.

\section{Radial wave equation with nonlocal potential}\label{sec4}
The radial wave equation
$$
{\partial_t^2}u(x,t)=\partial_x^2u(x,t)-\frac{k(k+1)}{x^2}u(x,t)-(Uu)(x,t), \qquad
\ k\in{\mathbb N}\cup{0}, \quad x\geq{0} 
$$
 with nonlocal potential
$$
(Uu)(x,t)={f}(x)\int_0^{\infty}{f}(\tau)u(\tau,t)d\tau,
$$
where a real function $f$ belongs to $L_2({\mathbb R}_+)$, 
can be rewritten as \eqref{e2} where 
\begin{equation}\label{AGH38}
Lu=l(u)=-\frac{d^2}{dx^2}u(x)+\frac{k(k+1)}{x^2}u(x)+{f}(x)(u, f)_{L_
2({\mathbb R}_+)} 
\end{equation}  with  domain
of definition
$\mathcal{D}(L)=\{u\in{W}_2^2({\mathbb R}_+)  \ | \ u(0)=0\}$ if 
 $k=0$ and
$$
\mathcal{D}(L)=\{u\in{L_2({\mathbb R}_+)} \ | \  l(u)\in{L_2({\mathbb
R}_+)}\},  \qquad k\in{\mathbb N}
$$ 
is a positive self-adjoint operator in $\mathfrak{H}=L_2({\mathbb R}_+)$. 

Denote by
$$
(\Gamma_{k+1/2}{u})(x)=\int^\infty_0\sqrt{sx}u(s)J_{k+1/2}(sx)ds 
$$ 
the Hankel transformation ($J_{k+1/2}(\cdot)$ is the Bessel function). It is known \cite[p. 545]{AG} that the
Hankel transformation determines a unitary and self-adjoint operator in $L_2({\mathbb R}_+)$ and
\begin{equation}\label{AGH111}
\Gamma_{k+1/2}\left(-\frac{d^2}{dx^2}+\frac{k(k+1)}{x^2}\right)u=x^2\Gamma_{k+1/2}{u}, \qquad u\in\mathcal{D}(L).
\end{equation}

Consider the unitary operator $X=F_{sin}\Gamma_{k+1/2}$ in  $L_2({\mathbb R}_+)$, where
$$
(F_{sin}{u})(\delta)=\sqrt{\frac{2}{\pi}}\int^\infty_0u(x)\sin{x\delta}dx
$$ 
is the sine-Fourier transformation in $L_2({\mathbb R}_+)$.

It is clear that 
\begin{equation}\label{AGH87}
\mathcal{B}=X^{-1}i\frac{d}{ds}X, \qquad  \mathcal{D}(\mathcal{B})=X^{-1}\{u\in{{W}_{2}^{1}}({\mathbb R}_+) : u(0)=0\}
\end{equation}
 is  a simple maximal symmetric operator in $L_2({\mathbb R}_+)$. Moreover, 
taking \eqref{AGH111} into account, we obtain
\begin{equation}\label{AGH58}
\left(-\frac{d^2}{dx^2}+\frac{k(k+1)}{x^2}\right)u=\mathcal{B}^*\mathcal{B}u,  \qquad u\in\mathcal{D}(L).
\end{equation}

By analogy with \eqref{AGH2} we define 
$\mathcal{T}=(\mathcal{B}-iI)(\mathcal{B}+iI)^{-1}.$  It is clear that $\mathcal{T}$ is  
 a unilateral shift in $L_2({\mathbb R}_+)$.
 A function $f\in{L_2({\mathbb R}_+)}$ is called \emph{non-cyclic} for the operator $\mathcal{T}^*$ if
$$
E_{f}:=\bigvee_{n=0}^{\infty}{\mathcal{T}^*}^n{f}
$$
is a proper subspace  of $L_2({\mathbb R}_+)$.

\begin{lemma}[\cite{MFAT1}]\label{AGH29}
If $f$ is non-cyclic for $\mathcal{T}^*$, then 
 the group of solutions of the Cauchy problem of \eqref{e2}
with the operator $L$ defined by \eqref{AGH38} has outgoing/incoming subspaces $D_\pm$
satisfying the conditions $(i), (ii), (iv)$ and $(v)$.
\end{lemma}
\begin{proof} 
By virtue of Theorem \ref{AGH10} it is sufficient to establish 
the existence of a simple maximal symmetric operator $B$ acting in a
subspace $\mathfrak{H}_0\subseteq{L_2({\mathbb R}_+)}$ and
 such that the operator $L$ in \eqref{AGH38} is  an extension of $B^2$.
Then the required subspaces $D_\pm$ are determined by \eqref{e2b}.

Denote $Y=FX$, where $F$ is the Fourier transformation in $L_2(\mathbb{R})$  (we consider $L_2({\mathbb R}_+)$ as a 
subspace of $L_2(\mathbb{R})$).  It is easy to see that $Y$ isometrically maps $L_2({\mathbb R}_+)$ onto $H^2(\mathbb{C}_+)$
and $Y\mathcal{B}=\delta{Y}$. Hence,  $Y\mathcal{T}=\frac{\delta-i}{\delta+i}Y$.

Let $f$ be non-cyclic for $\mathcal{T}^*$. Then $\mathfrak{H}_0=L_2({\mathbb R}_+)\ominus{E_f}$ 
 should be invariant with respect to $\mathcal{T}$. This means that
the subspace $Y\mathfrak{H}_0$ of ${H^2(\mathbb{C}_+)}$ turns out to be invariant with respect to the multiplication by  
$\frac{\delta-i}{\delta+i}$. By the Beurling theorem, there exists an inner function $\psi_0\in{H}^\infty({\mathbb{C}_+})$
such that $Y\mathfrak{H}_0=\psi_0{H}^2({\mathbb{C}_+})$.  Therefore, $\mathfrak{H}_0=\mathcal{V}L_2(\mathbb{R}_+)$, where
the isomeric operator $\mathcal{V}=Y^{-1}\psi_0(\delta)Y=\psi_0(\mathcal{B})$ commutes with $\mathcal{B}$.
Define, by analogy with \eqref{e7} and \eqref{new1},  a simple maximal symmetric operator $B$ acting in ${\mathfrak H}_0$:
\begin{equation}\label{e7b}
 Bu=\mathcal{B}u, \qquad \mathcal{D}(B)=\mathcal{D}(\mathcal{B})\cap\mathfrak{H}_0.
\end{equation}
  It follows from \eqref{AGH58} and \eqref{e7b}  that for all ${u}\in{D(B^2)}$
$$
Lu=\left(-\frac{d^2}{dx^2}+\frac{k(k+1)}{x^2}\right)u=\mathcal{B}^*\mathcal{B}u=\mathcal{B}^*Bu=\mathcal{B}Bu=B^2u.
$$
Therefore, $L$ is a positive self-adjoint extension of the symmetric operator $B^2$ acting in the subspace $\mathfrak{H}_0\subset{L_2(\mathbb{R}_+)}$.
\end{proof}

\begin{remark}\label{AGH44}
There is a natural relationship between the inner function $\psi_0$ which determines the subspace $\mathfrak{H}_0$ (and the subspaces $D_\pm$)  in the proof of Lemma \ref{AGH29}  
and the non-cyclic function $f$. Indeed, the function $Yf$ belongs to ${H}^2({\mathbb{C}_+})$ and it is non-cyclic for the adjoint of multiplication   
by $\frac{\delta-i}{\delta+i}$. Using the standard isometric mapping  of $H^2(\mathbb{D})$ onto $H^2(\mathbb{C})$ (see, e.g. \cite[p. 143]{Nik})
\begin{equation}\label{e51b}
\Phi  :  \gamma(e^{i\theta}) \to \frac{1}{\sqrt{\pi}(\delta+i)}\gamma\left(\frac{\delta-i}{\delta+i}\right) \qquad  {\gamma}(e^{i\theta})\in{H^2(\mathbb{D})}
\end{equation}
we conclude that $\gamma(e^{i\theta})=\Phi^{-1}Yf$ is a non-cyclic vector for the backward shift operator in $H^2(\mathbb{D})$.
According to \cite[Theorem 3.1.5]{DSS}, there exists $g\in{H^2(\mathbb{D})}$ and an inner function $\phi$ such that
\begin{equation}\label{AGH48}
\gamma(e^{i\theta})=[\Phi^{-1}Yf](e^{i\theta})=\overline{e^{i\theta}g(e^{i\theta})}\phi(e^{i\theta}).
\end{equation}

The functions $g$ and $\phi$ in \eqref{AGH48} are uniquely determined if $\phi$ is a normalized inner function \cite[Definition 3.1.4]{DSS}, which is 
relatively prime to the inner factor of $g$. In this case: 
$$
H^2(\mathbb{D})\ominus{\Phi^{-1}Y}E_f=\phi{H^2(\mathbb{D})}.
$$
 This means that $\psi_0(\delta)H^2(\mathbb{C}_+)=\Phi\varphi(e^{i\theta})\Phi^{-1}H^2(\mathbb{C}_+)$. Therefore, 
 $\psi_0(\delta)=\phi\left(\frac{\delta-i}{\delta+i}\right)$.
\end{remark}

As follows from the proof of Lemma \ref{AGH29} and Remark \ref{AGH44},  the outgoing/incoming subspaces $D_\pm$ 
are determined by the inner function  $\psi_0(\delta)=\phi\left(\frac{\delta-i}{\delta+i}\right)$, where $\phi$ is taken from the decomposition
\eqref{AGH48}.  The function $\phi$ is called \emph{associated inner function} of $\gamma=\Phi^{-1}Yf$
and it is uniquely determined by $\gamma$ in \eqref{AGH48}. 
 For this reason we can expect that the singularities of the Lax-Phillips scattering matrix
$S(\delta)$ associated with subspaces $D_{\pm}$ characterize the impact of nonlocal potential  $f(\cdot, f)$ in \eqref{AGH38}.

Let  an inner function $\psi_1$ be divisible by $\psi_0$.  Then $\psi_1=\psi_0\psi$, where $\psi$ is an inner function and
 $\psi_1{H}^2({\mathbb{C}_+})\subset\psi_0{H}^2({\mathbb{C}_+})$ \cite[p. 24]{Nik}.  This means that
 $Y^{-1}\psi_1{H}^2({\mathbb{C}_+})$ is a subspace of $\mathfrak{H}_0$ and
 $Y^{-1}\psi_1{H}^2({\mathbb{C}_+})=\mathfrak{H}_0^V$, where   
 ${V}=Y^{-1}\psi{Y}$ is an isomeric operator in $\mathfrak{H}_0$ which anticommutes with $B$. 

The operator $B_V$ defined by \eqref{e7} is simple maximal symmetric in $\mathfrak{H}_0^V$
and the  subspaces $D_\pm^V\subset{D_\pm}$ determined by \eqref{e2b} (with  $B_V$ instead of $B$)
 are outgoing/incoming for the group of solutions of the Cauchy problem of \eqref{e2}
with the operator $L$ defined by \eqref{AGH38}.

The pairs $D_\pm$ and  $D_\pm^V$ are asymptotically equivalent (Proposition \ref{AGH55b})
and the corresponding Lax-Phillips scattering matrices $S(\cdot)$ and $S_V(\cdot)$ are related 
in accordance with \eqref{e44}.  By virtue of \eqref{e44c}, the set of singularities of $S_V(\cdot)$ may
involve additional points  generated by zeros of the function  $\psi$ in $\mathbb{C}_+$ which have no relation with
the nonlocal potential  $f(\cdot, f)$.
  
 \begin{example}
 Consider the operator 
 \begin{equation}\label{AGH792}
Lu=-\frac{d^2}{dx^2}u(x)+e^{-x}P_{m}(x)\int_0^{\infty}e^{-\tau}P_{m}(
\tau)u(\tau)d\tau, 
\end{equation}
where $P_{m}$ is a real polynomial of order  $m$ with the domain
$\mathcal{D}(L)=\{u\in{W}_2^2({\mathbb R}_+)  \ | \ u(0)=0\}$.
The operator $L$ is positive self-adjoint in $L_2(\mathbb{R_+})$ and it can be defined by 
\eqref{AGH38} with $f=e^{-x}P_{m}(x)$ and $k=0$. In this case,
$X=F_{sin}\Gamma_{1/2}=F_{sin}^2=I$  and the operator $\mathcal{B}$ in
\eqref{AGH87} coincides with $i\frac{d}{dx}$,
$\mathcal{D}(i\frac{d}{dx})=\{u\in{{W}_{2}^{1}}({\mathbb R}_+) : u(0)=0\}$. 
Therefore, $\mathcal{T}=(i\frac{d}{dx}-iI)(i\frac{d}{dx}+iI)^{-1}.$
 
Each function  $e^{-x}P_{m}(x)$ can be presented as
\begin{equation}\label{AGH123}
e^{-x}P_{m}(x)=\sum_{n=0}^{m}\alpha_n{q}_{n}(2x), \qquad \alpha_m\not=0, \quad \alpha_n\in\mathbb{R},
\end{equation}
 where  
 $$
 {q}_n(x)=\frac{e^{x/2}}{n!}\frac{d^n}{dx^n}(x^ne^{-x}), \qquad
n=0,1\ldots
$$
are the Laguerre functions. Using the well-known relation  $\mathcal{T}{q}_n(2x)={q}_{n+1}(2x)$ 
 \cite[p. 363]{AG} and taking into account that the functions $\{q_n\}$ form
 an orthonormal basis of $L_2({\mathbb R}_+)$, 
 we obtain that $f=e^{-x}P_{m}(x)$ is orthogonal to  the subspace $\mathcal{T}^{m+1}L_2({\mathbb R}_+)$.
 Obviously, $E_f$ is also orthogonal to this subspace and the vector $f$ is non-cyclic for 
 $\mathcal{T}^*$.
 
  Due to Lemma \ref{AGH29}, the wave equation
 \eqref{e2} with the operator $L$ defined by \eqref{AGH792} has  subspaces $D_\pm$
satisfying the conditions $(i), (ii), (iv)$ and $(v)$. Such subspaces are not determined uniquely.
By virtue of Remark \ref{AGH44},  the `largest' subspaces $D_\pm$ are determined  
by the function $\psi_0(\delta)=\phi\left(\frac{\delta-i}{\delta+i}\right)$, where $\phi$  is  the associated inner function of $\gamma=\Phi^{-1}Yf$ 
from the decomposition \eqref{AGH48}.
 Here, $\Phi$ is the isometric mapping of $H^2(\mathbb{D})$ onto $H^2(\mathbb{C}_+)$,
 see \eqref{e51b} and, in our case, the operator $Y=FX$ is reduced to the Fourier transformation $F$.

Using relation (25) in \cite[p. 158]{Bat},  and taking into account that $q_n(x)=e^{-\frac{x}{2}}L_n(x)$,
where $L_n(x)$ is the Laguerre polynomial,  we get
$$
\Phi^{-1}Y{q}_n(2x)=\frac{1}{2\sqrt{2\pi}}\Phi^{-1}\int_0^\infty{L_n(t)e^{-\frac{1}{2}(1-i\delta)t}}dt=\frac{i}{\sqrt{2\pi}}\Phi^{-1}\frac{(\delta-i)^n}{(\delta+i)^{n
+1}}=\frac{i}{\sqrt{2}}e^{in\theta}.
$$
Therefore,
$$
\gamma(e^{i\theta})=\Phi^{-1}Y[e^{-x}P_{m}(x)]=\frac{i}{\sqrt{2}}\sum_{n=0}^{m}{\alpha_n}e^{in\theta}.
$$
Substituting the obtained expression to the left-hand side of \eqref{AGH48} we arrive at the conclusion that
$$
g(e^{i\theta})=-\frac{i}{\sqrt{2}}\sum_{n=0}^{m}{{\alpha_n}}e^{i(m-n)\theta} \quad
\mbox{and} \quad \phi(e^{i\theta})=e^{i(m+1)\theta}.
$$  Therefore, $\psi_0(\delta)=\left(\frac{\delta-i}{\delta+i}\right)^{m+1}$.

It follows from the proof of Lemma \ref{AGH29} that the required subspaces $D_\pm$ are determined by \eqref{e2b},
where 
$$
B=i\frac{d}{dx},  \quad \mathcal{D}(B)=\{u\in{{W}_{2}^{1}}({\mathbb R}_+) : u(0)=0\}\cap\mathfrak{H}_0
$$
is a simple maximal symmetric operator in the Hilbert space $\mathfrak{H}_0=\mathcal{T}^{m+1}L_2({\mathbb R}_+)$.
 \end{example}

\section{Appendix. Functional calculus for simple maximal symmetric operator $B$}\label{sec5}
Since $B$ is a maximal symmetric operator in $\mathfrak{H}_0$
its spectral function  $E_\delta$ is determined uniquely (see \cite[$\S$ 111]{AG}
for the definition of spectral functions of symmetric operators; 
the uniqueness of $E_\delta$ follows from \cite[$\S$ 112]{AG}).

In contrast to the case of self-adjoint operators, the spectral function
is not orthogonal, i.e., $E_\delta$ can not be an orthogonal projection operator in $\mathfrak{H}_0$
 and  $E_{s}E_{r}\not=E_{p}$, where $p=\min\{s, r\}$.
Therefore, the standard functional calculus for self-adjoint operators can not be used.
However, taking into account the uniqueness of $E_\delta$ for a given $B$, 
it is natural to expect that an analog of functional calculus for $B$ 
with properties of the conventional
 functional calculus for self-adjoint operators
can be developed.
We restrict our attention to functions from $H^{\infty}(\mathbb{C_+})$.

\subsection{Functional calculus.} 
To the best of our knowledge, the functional calculus for maximal symmetric operators 
 was firstly developed by Plesner in series of short papers \cite{Ples1}  -\cite{Ples3}.  
He mentioned that the integral $\int_{-\infty}^{\infty}\psi(\delta)dE_\delta{f}$
has sense for functions $\psi$ from the so-called `narrow' class $\Omega$ of analytic functions
 in $\mathbb{C}_+$  (actually $\Omega$ contains each Hardy class
$H^p(\mathbb{C}_+)$, \ ${p}\geq{1}$).  For this reason the operator $\psi(B)$ is defined as follows:
 $$
\psi(B)f=\int_{-\infty}^{\infty}\psi(\delta)dE_\delta{f}.
$$
The equivalent definition of $\psi(B)$  in terms of sesquilinear forms:
 \begin{equation}\label{e44d} 
(\psi(B)f, g)=\int_{-\infty}^{\infty}\psi(\delta)d(E_\delta{f}, g),  \qquad  \forall{g}\in\mathfrak{H}_0.
\end{equation}

Let $A$ be a self-adjoint extension of $B$ acting in a Hilbert space ${\mathfrak{H}}\supset\mathfrak{H}_0$ and 
let ${E}_\delta^A$ be its orthogonal spectral function. Then 
$$
\psi(A)=\int_{-\infty}^{\infty}\psi(\delta)d{E}_\delta^A, \qquad \psi{\in}H^{\infty}(\mathbb{C_+})
$$
is a bounded operator in ${\mathfrak{H}}$.  Taking into account
that $E_\delta=P{E}_\delta^A$, where $P$ is the orthogonal projection in ${\mathfrak{H}}$ on $\mathfrak{H}_0$
and using \eqref{e44d} we obtain
$$
(P\psi(A)f, g)=(\psi(A)f, g)=\int_{-\infty}^{\infty}\psi(\delta)d({E}_\delta^A{f}, g)=(\psi(B)f, g), \quad f, g \in\mathfrak{H}_0.
$$
Therefore, 
\begin{equation}\label{e49}
\psi(B)f=P\psi(A)f, \qquad \psi{\in}H^{\infty}(\mathbb{C_+}),  \quad  f\in\mathfrak{H}_0.
\end{equation}

The formula \eqref{e49} does not depend on the choice of self-adjoint extension $A$ and it can be used as the definition of $\psi(B)$.

Actually \eqref{e49} allows one to define $\psi(B)$ for wider classes of functions $\psi$ (not necessarily in $H^{\infty}(\mathbb{C}_+)$). 
 However, if $\psi\in{H^{\infty}(\mathbb{C}_+)}$,  
 the formula \eqref{e49}  can be simplified. To that end, in addition to the given operator $B$ in $\mathfrak{H}_0$ with nonzero defect number $m$ in $\mathbb{C}_-$,  we
consider a simple symmetric operator $B'$  in a Hilbert space  $\mathfrak{H}'_0$  with 
the nonzero defect number $m$ in $\mathbb{C}_+$.

 By virtue of \eqref{e14} and \eqref{e14b} there exists a unitary mapping 
 $\Xi=\Xi_-\oplus\Xi_+$ that maps
  ${\mathfrak{H}}=\mathfrak{H}'_0\oplus\mathfrak{H}_0$ onto  $L_2(\mathbb{R}, N)$ and  
  $$
  \Xi[B'\oplus{B}]=i\frac{d}{dx}\Xi,
  $$
   where
  $i\frac{d}{dx}$ is a symmetric operator in $L_2(\mathbb{R}, N)$ with defect number $m$ in $\mathbb{C}_\pm$ and 
  the domain $\mathcal{D}(i\frac{d}{dx})=\{u\in{{W}_{2}^{1}}({\mathbb R}, N) : u(0)=0\}$.
  Denote
 \begin{equation}\label{AGH34}
A=\Xi^{-1}i\frac{d}{dx}\Xi,    
\end{equation}
where $i\frac{d}{dx}$, $\mathcal{D}(i\frac{d}{dx})={{W}_{2}^{1}}({\mathbb R}, N)$ is a self-adjoint extension in $L_2(\mathbb{R}, N)$
of the symmetric operator above.  It is clear that  $A$ is a self-adjoint operator in ${\mathfrak{H}}$
and $A$ is an extension of $B$.
 
 \begin{lemma}\label{AGH214}
  Let $\psi\in{H^{\infty}(\mathbb{C}_+)}$.
 Then $\psi(B)f=\psi(A)f$  ($f\in\mathfrak{H}_0$) for the self-adjoint operator $A$ defined in \eqref{AGH34}.
  \end{lemma}
\begin{proof}
According to the Paley-Wiener theorem, the operator
\begin{equation}\label{AGH1}
 [{\mathsf Y}\gamma](\delta):=F\Xi\gamma=\frac{1}{\sqrt{2\pi}}\int^\infty_{-\infty}e^{i\delta{s}}(\Xi\gamma)(s)ds, \qquad \forall{\gamma}\in{\mathfrak{H}}=\mathfrak{H}'_0\oplus\mathfrak{H}_0 
 \end{equation}
 maps isometrically ${\mathfrak{H}}$ onto $L_2(\mathbb{R}, N)$.  Moreover, under this mapping, the self-adjoint operator $A$ is transformed 
 to the operator of multiplication by $\delta$  in $L_2(\mathbb{R}, N)$. Therefore,
\begin{equation}\label{new56}
\psi(A)={\mathsf Y}^{-1}\psi(\delta){\mathsf Y}=\Xi^{-1}F^{-1}\psi(\delta)F\Xi.
\end{equation}

We note that $F^{-1}\psi(\delta)FL_2(\mathbb{R}_+, N)\subset{L_2(\mathbb{R}_+, N)}$ 
(this inclusion follows from the Paley-Winer theorem and the fact that 
$\psi(\delta)H^2(\mathbb{C}_+)\subset{H^2(\mathbb{C}_+)}$).
This relation and \eqref{e49} lead to the conclusion that 
$$
\psi(A)f=\Xi^{-1}F^{-1}\psi(\delta)F\Xi{f}=P\Xi^{-1}F^{-1}\psi(\delta)F\Xi{f}=P\psi(A)f=\psi(B)f
$$
for all $f\in\mathfrak{H}_0$.
\end{proof}

\begin{corollary}\label{new4}
The following statements are true:
\begin{itemize}
\item[(a)] if $\psi_i\in{H^{\infty}(\mathbb{C}_+)}$, then 
$\psi_1(B)\psi_2(B)=\psi_2(B)\psi_1(B)$; \vspace{3mm}
\item[(b)] if $\psi\in{H^{\infty}(\mathbb{C}_+)}$  is an inner function,
then $\psi(B)$ is an isometric operator in $\mathfrak{H}_0$; \vspace{3mm}
\item[(c)] if $\psi=\frac{1}{\delta+i}$, then $\psi(B)=(B+iI)^{-1}$; \vspace{3mm}
\item[(d)] if $\psi\in{H^{\infty}(\mathbb{C}_+)}$, then  
$B\psi(B)=\psi(B)B$; \vspace{3mm}
\end{itemize}
\end{corollary}
\begin{proof}
Items (a) - (c)  follow from Lemma \ref{AGH214}. Assuming in (a) that $\psi_1=\frac{1}{\delta+i}$ and $\psi_2=\psi$,
we get $(B+iI)^{-1}\psi(B)=\psi(B)(B+iI)^{-1}$ that yields (d).
\end{proof}
 
 \subsection{Relationship with contraction operators.}
Another approach to the definition of $\psi(B)$ deals with the Cayley transform $T$ of $B$.  
The unilateral shift $T$ determined by \eqref{AGH2} is an example of completely nonunitary contraction in $\mathfrak{H}_0$. 
For such kind of operators,  the functional calculus is well developed \cite[Chapter III]{SNK}.
Below we outline some facts important for our presentation.

Let $A$ be a self-adjoint extension of $B$  defined by \eqref{AGH34}. 
Its  Cayley transform
$$
W=(A-iI)(A+iI)^{-1}
$$
acts in ${\mathfrak H}$ and it is a minimal unitary dilation of $T$. 
According to \cite[p. 117]{SNK}
$$
\phi(T)=P\phi(W),  \qquad  \phi\in{H^\infty(\mathbb{D})}, \quad \mathbb{D}=\{\lambda\in\mathbb{C} : |\lambda|<1\},
$$
where $P$ is the orthogonal projection operator in ${\mathfrak{H}}$ on $\mathfrak{H}_0$.

The spectral function ${E}_\delta^A$ of $A$ and 
the spectral function ${E}_\theta$   of $W$
are closely related \cite[$\S$ 79]{AG}: 
$$
{E}_\delta^A={E}_\theta,  \quad \mbox{where} \quad \delta=i\frac{1+e^{i\theta}}{1-e^{i\theta}}=i\frac{e^{-i\theta/2}+e^{i\theta/2}}{e^{-i\theta/2}-e^{i\theta/2}}=-\cot\frac{\theta}{2}, \quad \theta\in[0, 2\pi].
$$
Moreover,  by virtue of \cite[p. 138]{Birman} 
\begin{equation}\label{e51}
\phi(W)=\int_0^{2\pi}\phi(e^{i\theta})d{E}_\theta=\int_{-\infty}^{\infty}\phi\left(\frac{\delta-i}{\delta+i}\right)d{E}_\delta^A=\psi(A), 
\end{equation} 
where $\psi(\delta)=\phi(\frac{\delta-i}{\delta+i})$ belongs to $H^\infty(\mathbb{C}_+)$.

Using Lemma \ref{AGH214} and \eqref{e51} we arrive at the conclusion that
\begin{equation}\label{e52}
\phi(T)f=P\phi(W)f=P\psi(A)f=\psi(B)f, \qquad  f\in\mathfrak{H}_0, 
\end{equation}
where $\psi(\delta)=\phi(\frac{\delta-i}{\delta+i}){\in}H^\infty(\mathbb{C}_+)$.
The obtained relationship allows one to reduce the investigation of  $\psi(B)$  to the investigation of
$\phi(T)$.  

\bigskip

{\bf Acknowledgments.} 
The authors gratefully acknowledge support from the Polish Ministry of Science and
Higher Education.

\bibliographystyle{amsplain}

\end{document}